\documentclass[11pt]{amsart}

\usepackage[a4paper,hmargin=3.5cm,vmargin=4cm]{geometry}
\usepackage{amsfonts,amssymb,amscd,amstext}
\usepackage{graphicx}
\usepackage[dvips]{epsfig}
\usepackage{color}

\usepackage{fancyhdr}
\usepackage{verbatim}
\usepackage[utf8]{inputenc}
\usepackage{hyperref}

\input xy
\xyoption{all}

\usepackage{times}
\usepackage{enumerate}
\usepackage{titlesec}
\usepackage{mathrsfs}

\titleformat{\section}%[display]
{\filcenter\bfseries\large} {\thesection{.}}{0.2cm}{}%[$\vspace*{-1.0cm}$]
\titleformat{\subsection}[runin]
{\bfseries} {\thesubsection{.}}{0.15cm}{}[.]
\titleformat{\subsubsection}[runin]
{\em}{\thesubsubsection{.}}{0.15cm}{}[.]
\usepackage[up,bf]{caption}
\newtheorem{theorem}{Theorem}[section]

\newtheorem{lemma}[theorem]{Lemma}

\theoremstyle{definition}

\newtheorem{problem}[theorem]{Problem}
\newtheorem{example}[theorem]{Example}

\numberwithin{equation}{section}
\numberwithin{figure}{section}

\newcommand\C{\mathbb{C}}
\newcommand\D{\overline{\mathbb D}}
\newcommand\CP{\mathbb{CP}}
\renewcommand\D{\mathbb D}

\newcommand\R{\mathbb{R}}

\newcommand\igot{\mathfrak{i}}

\renewcommand\igot{\mathfrak{i}}

\newcommand\ggot{\mathfrak{g}}

\renewcommand\imath{\igot}

\newcommand\di{\partial}

\begin{document}

\fancyhead[LO]{Zero-curvature point of  minimal graphs}
\fancyhead[RE]{ D. Kalaj}
\fancyhead[RO,LE]{\thepage}

\thispagestyle{empty}

%% Title
\vspace*{1cm}
\begin{center}
{\bf\LARGE Zero-curvature points of minimal graphs}

\vspace*{0.5cm}

%% Authors
{\large\bf   David Kalaj}
\end{center}

%% Addresses and finantial support
%\footnote[0]{\vspace*{-0.4cm}
%}
%% Abstract, keywords, and MSC

\vspace*{1cm}

\begin{quote}
{\small
\noindent {\bf Abstract}\hspace*{0.1cm}
Motivated by a classical result of Finn and Osserman (1964), who proved that the Scherk surface over the square inscribed in the unit disk is extremal for the Gaussian curvature of the point $O$ (so-called \emph{centre}) of the minimal graphs above the center $0$  of unit  unit disk, provided the tangent plane is horizontal, we ask and answer to the question concerned the extremal of "second derivative" of the Gaussian curvature of such graphs provided that its curvature at $O$ is zero. We prove that the extremals are certain Scherk type minimal surfaces over the regular hexagon inscribed in the unit disk, provided that the Gaussian curvature vanishes and the tangent plane is horizontal at the centre.

\vspace*{0.2cm}

\noindent{\bf Keywords}\hspace*{0.1cm} conformal minimal surface, minimal graph, curvature

\vspace*{0.1cm}

\noindent{\bf MSC (2010):}\hspace*{0.1cm} 53A10, 32B15, 32E30, 32H02}
% 53D10  contact, Legendrian
%  37J55 Contact systems
%  32E30 Runge approx.
%  32\mathcal{H}02 holom mappings, embeddings,...

\vspace*{0.1cm}
\noindent{\bf Date: \today} %\rm June 17, 2017. This version: \today}
\end{quote}

\vspace{0.2cm}

%%%%%%%%%%
%%%%%%%%%%
%%%%%%%%%%
%%%%%%%%%%
%%%%%%%%%%
%%%%%%%%%%

\section{Introduction}
\label{sec:intro}

Nonparametric minimal graph in $\mathbf{R}^3$ over a domain $D\subset \mathbf{C}\cong \mathbf{R}^2$ is  given by
$$S=\{(u,v,\mathbf{f}(u,v)):(u,v)\in D\},$$ where $\mathbf{f}$ is a solution of minimal surface equation: $$(1+\mathbf{f}_u^2)\mathbf{f}_{vv}-2\mathbf{f}_u\mathbf{f}_v\mathbf{f}_{uv}+(1+\mathbf{f}_v^2)\mathbf{f}_{uu}=0.$$

Let $M\subset \R^3=\C\times \R$ be a minimal graph lying over the unit disc $\D\subset \C$.
Let $x=(x_1,x_2,x_3):\D\to M$ be a conformal harmonic parameterization of $M$ with $x(0)=0$.
Its projection $(x_1,x_2):\D\to \D$ is a harmonic diffeomorphism of the disc which may be assumed
to preserve the orientation. Let $z$ be the complex variable in $\D$, and write
$x_1+\imath x_2 = f$ in the complex notation.
We denote by $f_z=\di f/\di z$ and $f_{\bar z}=\di f/\di \bar z$ the Wirtinger derivatives of $f$.
The function $\omega$ defined by
\begin{equation}\label{eq:omega}
	\overline{f_{\bar z}} = \omega f_z
\end{equation}
is called the {\em second Beltrami coefficient} of $f$, and the above is the
{\em second Beltrami coefficient} with $f$ as a solution.

Orientability of $f$ is equivalent to $J(f,z)=|f_z|^2-|f_{\bar z}|^2>0$, hence to
$|\omega|<1$ on $\D$. Furthermore, the function $\omega$ is holomorphic whenever
$f$ is harmonic and orientation preserving. (In general, it is meromorphic when $f$ is harmonic.)
To see this, let
\begin{equation}\label{eq:fhg}
	u+\imath v = f = h+\overline g
\end{equation}
be the canonical decomposition of the harmonic map $f:\D\to\D$,
where $h$ and $g$ are holomorphic functions on the disc. Then,
\begin{equation}\label{eq:omega2}
	f_z=h',\quad\ f_{\bar z}=\overline g_{\bar z}= \overline{g'}, \quad\
	\omega =  \overline{f_{\bar z}}/f_z = g'/h'.
\end{equation}
In particular, the second Beltrami coefficient $\omega$ equals the meromorphic function $g'/h'$
on $\D$. In our case we have $|\omega|<1$, so it is holomorphic map $\omega:\D\to\D$.

We now consider the Enneper--Weierstrass representation of the minimal graph
$\mathbf{w}=(u,v,t):\D \to M\subset \D\times \R$ over $f$, following Duren \cite[p.\ 183]{Duren2004}. We have
\begin{eqnarray*}
	u(z) &=& \Re f(z) = \Re \int_0^z \phi_1(\zeta)d\zeta \\
	v(z) &=& \Im f(z) = \Re \int_0^z \phi_2(\zeta)d\zeta \\
	t(z) &=& \Re \int_0^z \phi_3(\zeta)d\zeta
\end{eqnarray*}
where
\begin{eqnarray*}
	\phi_1 &=& 2(u)_z = 2(\Re f)_z = (h+\bar g + \bar h + g)_z = h'+g', \\
	\phi_2 &=& 2(v)_z = 2(\Im f)_z = \imath(\bar h+g - h -\bar g)_z = \imath(g'-h'), \\
	\phi_3 &=& 2(t)_z = \sqrt{-\phi_1^2-\phi_2^2} = \pm 2\imath \sqrt{h'g'}.
\end{eqnarray*}
The last equation follows from the identity $\phi_1^2+\phi_2^2+\phi_3^2=0$ which
is satisfied by the Enneper--Weierstrass datum $\phi=(\phi_1,\phi_2,\phi_3)=2\di \mathbf{w}$
of any conformal minimal (equivalently, conformal harmonic) immersion $\mathbf{w}:D\to\R^3$
from a conformal surface $D$. Let us introduce the notation $p=f_z$. We have that
\begin{equation}\label{eq:p}
	p = f_z = (\Re f)_z + \imath (\Im f)_z = \frac12(h'+g' + h'-g') = h'.
\end{equation}
By using also $\omega =  \overline{f_{\bar z}}/f_z = g'/h'$ (see \eqref{eq:omega2}), it follows that
\[
	\phi_1 = h'+g'=p(1+\omega),\quad \phi_2 = -\imath(h'-g')=-\imath p(1-\omega),\quad
	\phi_3 = \pm 2\imath p \sqrt{\omega}.
\]
From the formula for $\phi_3$ we infer that $\omega$ has a well-defined holomorphic square root:
\begin{equation}\label{eq:q}
	\omega = q^2,\qquad q:\D\to \D\ \ \text{holomorphic}.
\end{equation}
% (See Duren \cite[Theorem, p.\ 177]{Duren2004}.)
In terms of the Enepper--Weierstrass parameters $(p,q)$ given by \eqref{eq:p} and \eqref{eq:q} we obtain
\begin{equation}\label{eq:EW}
	\phi_1 = p(1+q^2),\quad \phi_2 = -\imath p(1-q^2),\quad
	\phi_3 = -2\imath p q.
\end{equation}
(The choice of sign in $\phi_3$ is a matter of convenience; since we have two choices of sign for
$q$ in \eqref{eq:q}, this does no cause any loss of generality.) Hence,
\[
	\mathbf{w}(z) = \left(\Re f(z), \Im f(z), \Im \int_0^z 2 p(t) q(t) dt \right),\quad z\in\D.
\]
Let $\ggot:\D\to\CP^1$ denote the Gauss map of the minimal graph $x$. It is defined up to
the choice of a stereographic projection of the unit sphere in $\R^3$ to $\R^2\times \{0\}\cong\C$.
We choose the projection from $(0,0,-1)$, which makes the plane $\R^2\times\{0\}$
with the upward orientation correspond to the origin $0\in \C\subset \CP^1$.
This choice of $\ggot$ is given by the formula
\begin{equation}\label{eq:Gaussmap}
	\ggot = \frac{\phi_1-\imath \phi_2}{\phi_3}=\frac{2pq^2}{-2\imath pq}= \imath q.
\end{equation}
(The formula used in \cite[Eq.\ (2.79)]{AlarconForstnericLopez2021}, which corresponds to the stereographic projection
from $(0,0,1)$, would give $\ggot=-\imath/q$.)

The curvature $\mathcal{K}$ of the minimal graph $M$ is expressed in terms of $(h,g,\omega)$ \eqref{eq:omega2}, and in terms
of the Enneper--Weierstrass parameters $(p,q)$, by
\begin{equation}\label{eq:curvatureformula}
	\mathcal{K}(f(z)) = - \frac{|\omega'|^2}{|h'g'|^2(1 + |\omega|)^4} = - \frac{4|q'|^2}{|p|^2(1 + |q|^2)^4},
\end{equation}
where $p=f_z$ and $\omega=q^2=\overline{f_{\bar z}}/f_z$. (See Duren \cite[p.\ 184]{Duren2004}.)

In order to motivate our problem let us recall  Heinz-Hopf-Finn--Osserman problem

\begin{problem}\label{probleme}
What is the supremum of $|\mathcal{K}(O)|$ evaluated at the point $O$ (which we will call \emph{centre}) above the center of the unit disk, over all minimal graphs lying over $\D$? Is
\begin{equation}\label{eq:FinnOsserman}
	|\mathcal{K}(O)|< \frac{\pi^2}{2}
\end{equation}
the precise upper bound?
\end{problem}

It was shown by Finn and Osserman \cite{FinnOsserman1964} in 1964 that the
upper bound in \eqref{eq:FinnOsserman} is indeed sharp if  $q(0)=0$, which means that
the tangent plane $T_0 M=\C\times\{0\}$ being horizontal (and hence $f$ is conformal at $0$).
Although there is no minimal graph lying over the whole unit disc $\D$ whose centre curvature equals
$\frac{\pi^2}{2}$, there is a sequence of minimal graphs whose centre curvatures converge to  $\frac{\pi^2}{2}$,
and the graphs converge to the Scherk surface lying over square inscribed into the unit disc.
The associated Beltrami coefficient of the Scherk surface is $\omega(z)=z^2$, with $q(z)=z$.
We refer to Duren \cite[p.\ 185]{Duren2004} for a survey of this subject.

So far the best inequality for graphs whose tangent planes are not horizontal has been given by R. Hall, who proved \cite{Hall1982} (1982)
\begin{equation}\label{eq:hall}
	|\mathcal{K}(O)|< \frac{16\pi^2}{27}
\end{equation}
He also in \cite{Hall1998} showed that the estimate \eqref{eq:hall} is not sharp by a very small numerical improvement.

In this paper we consider points of minimal surfaces that have zero Gaussian curvature. In that case, the gradient of the Gaussian curvature vanishes at that point, and we call those points stationary points of the Gaussian curvature. We will estimate the "second derivative" of Gaussian curvature at stationary points. \subsection{Zero Gaussian curvature points}
For a non-parametric surface $\zeta=\mathbf{f}(u,v)$, the Gaussian curvature of the surface at the point $\mathbf{w}=(u,v,\zeta)$ is given by $$\mathcal{K}(\mathbf{w}) = \frac{\mathbf{f}_{uu}\mathbf{f}_{vv}-\mathbf{f}_{uv}^2}{(1+\mathbf{f}_u^2+\mathbf{f}_v)^2}.$$ The aim of this paper is to study the behaviors of the surface near the the point $(u,v)$ of zero Gaussian curvature, i.e. near the point in which we have $$\mathbf{f}_{uu}\mathbf{f}_{vv}-\mathbf{f}_{uv}^2=0.$$ Notice that, it follows from \eqref{eq:curvatureformula} that the zero-points of a minimal surface are isolated.
We will show  in Lemma~\ref{mainlema} below that if $\mathcal{K}(0)=0$ then the following unrestricted limit exists $$\mathcal{K}''(O):=\lim_{w\to 0}\frac{\mathcal{K}(\mathbf{w})}{2(|w|^2+\left<\nabla \mathbf{f}(w),w\right>^2)}.$$

Now the following problem is natural
\begin{problem}\label{prob3}
What is the supremum of $|{\mathcal{K}''(O)}|$ over all minimal graphs lying over $\D$, provided that the Gaussian curvature is equal to zero in $0$.
\end{problem}

We first prove the following general result

\begin{theorem}\label{th:teo}
Let $S:\zeta=\mathbf{f}(u,v)$ be a minimal surface over the unit disk $\D$. Assume also at the point $O=(0,0,0)\in S$ the Gaussian curvature of $S$ vanishes. Then $$|\mathcal{K}''(O)|< \frac{256 \pi ^4}{729}.$$
\end{theorem}

Then we prove the following
\begin{theorem}[The main result]\label{mainresult}
Let $S=\{(u,v,\mathbf{f}(u,v)): (u,v)\in\D\}$ be a non-parametric minimal surface above the unit disk, with vanishing  Gaussian curvature $\mathcal{K}$   and with a horizontal tangent plane at the centre $O$. Then it exists $$-\mathcal{K}''(O):=\lim_{|w|\to 0}\frac{|\mathcal{K}(w)|}{|w|^2}$$ which is equal to
$2(\mathbf{f}^2_{uuu}+\mathbf{f}^2_{vvv})$
and there hold the sharp inequality
\begin{equation}\label{interesting}2(\mathbf{f}^2_{uuu}+\mathbf{f}^2_{vvv})< \frac{(2\pi)^4 }{3^4}. \end{equation} In other words, for every constant $C<\frac{(2\pi)^4 }{3^4}$, there is a minimal surface $S_1=\{(u,v,\mathbf{f}_1(u,v)): (u,v)\in\D\}$ over the unit disk with a zero curvature at the point above the center and horizontal tangent plane so that $-\mathcal{K}_1''(O)\ge C$. The equality in \eqref{interesting} is never attained.
\end{theorem}

\section{Proof of results}
We start the proofs by the following lemma
\begin{lemma}\label{mainlema}
Let $\zeta=\mathbf{f}(u,v)$ be a solution of minimal surface equation. Assume also at the point $O=(0,0,0)$,
$$\mathcal{K}(\mathbf{\mathbf{w}})=\frac{\mathbf{f}_{uu} \mathbf{f}_{vv}-\mathbf{f}_{uv}^2}{(1+\mathbf{f}_u^2+\mathbf{f}_v^2)^2}=0.$$ Then $\nabla\mathbf{f}(0)=0$, and the following (unrestricted)  limit exists
$$\mathcal{K}''(O):=\lim_{w\to 0}\frac{\mathcal{K}(\mathbf{w})}{2(|w|^2+\left<\nabla \mathbf{f}(w),w\right>^2)}$$ and is equal to the function $$-\frac{\left(1+\mathbf{f}_{u}^2\right)^3\mathbf{f}_{vvv}^2 +2 \mathbf{f}_{v} \mathbf{f}_{vvv} \mathbf{f}_{u}
 \left(3(\mathbf{f}_u^2+\mathbf{f}_{v}^2)-\mathbf{f}_u^2\mathbf{f}_v^2\right)\mathbf{f}_{uuu}+\left(1+\mathbf{f}_{v}^2\right)^3 \mathbf{f}_{uuu}^2}{\left(\left(1+\mathbf{f}_{v}^2+\mathbf{f}_{u}^2\right)^2 \left(1+\mathbf{f}_{u}^2+\mathbf{f}_{v}^2 \left(1-3 \mathbf{f}_{u}^2\right)\right)^2\right)}$$ evaluated in $w=(0,0)$. Here $\mathbf{w}=(w,\mathbf{f}(w))$.
Further
$\mathbf{f}_{uu}=0$, $\mathbf{f}_{vv}=0$ and $\mathbf{f}_{uv}=0$ and $\mathbf{f}_{uuu}, \mathbf{f}_{uuv}, \mathbf{f}_{uvv}, \mathbf{f}_{vvv}$ are uniquely determined by $\mathcal{K}''(O)$, $\mathbf{f}_u$ and $\mathbf{f}_v$ up to their sign at that point $w$.

In particular if at $w$, $\mathbf{f}_u=0$ and $\mathbf{f}_v=0$, then   \begin{equation}\label{h0}\mathcal{K}''(O)=-2(\mathbf{f}^2_{uuu}(0)+\mathbf{f}^2_{vvv}(0)).\end{equation}
\end{lemma}
\begin{proof}
If $\mathcal{K}(O)=0$ then at $w=(0,0)$ we have $$\mathbf{f}_{uu} \mathbf{f}_{vv}-\mathbf{f}_{uv}^2=0,$$ and $$\mathbf{f}_{uv}=\sqrt{\mathbf{f}_{uu} \mathbf{f}_{vv}}.$$ Moreover  $\mathbf{f}_{uu}$ and $\mathbf{f}_{vv}$ have the same sign. Assume, without loosing of generality  that both are non-negative.
Then we use the equation
$$\mathbf{f}_{uu} (1+\mathbf{f}_v^2)-2 \mathbf{f}_u \mathbf{f}_v \mathbf{f}_{uv}+\mathbf{f}_{vv} (1+\mathbf{f}_u^2)=0,$$ which becomes
$$\mathbf{f}_{uu}+\mathbf{f}_{vv}+(\sqrt{\mathbf{f}_{uu}}\mathbf{f}_v-\sqrt{\mathbf{f}_{vv}}\mathbf{f}_u)^2=0.$$
%(If they are  negative then we use the equation
%$$-\mathbf{f}_{uu} (1+\mathbf{f}_v^2)+2 \mathbf{f}_u \mathbf{f}_v \mathbf{f}_{uv}-\mathbf{f}_{vv} (1+\mathbf{f}_u^2)=0$$
%and thus
%$$-\mathbf{f}_{uu}-\mathbf{f}_{vv}+(\sqrt{-\mathbf{f}_{uu}}\mathbf{f}_v+\sqrt{-\mathbf{f}_{vv}}\mathbf{f}_u)^2=0.)$$
 Therefore
\begin{equation}\label{seconde}\mathbf{f}_{uu}(0)=\mathbf{f}_{vv}(0)=\mathbf{f}_{uv}(0)=0.\end{equation}

Now for $\mathbf{w}=(u,v,\mathbf{f}(u,v))=(w,\mathbf{f}(w))$ we have $$\mathcal{K}(\mathbf{w}) = \mathcal{K}(O)+\left<\nabla \mathcal{K}(O),w\right> +\frac{1}{2}\left<\mathcal{H}(O)w,w\right>+o(|w|^2).$$

Here
$$\mathcal{\mathcal{H}}=\left(
                                                                                                                                   \begin{array}{cc}
                                                                                                                                      \mathcal{K}_{uu} &  \mathcal{K}_{uv} \\
                                                                                                                                      \mathcal{K}_{uv} &  \mathcal{K}_{vv} \\
                                                                                                                                   \end{array}
                                                                                                                                  \right)$$

is the Hessian matrix of $\mathcal{K}$.
So

$$\mathcal{H}=\left(
                \begin{array}{cc}
                  \frac{-2 \mathbf{f}_{uuv}^2+2 \mathbf{f}_{uvv} \mathbf{f}_{uuu}}{\left(1+\mathbf{f}_v^2+\mathbf{f}_u^2\right)^2} & \frac{-\mathbf{f}_{uvv} \mathbf{f}_{uuv}+\mathbf{f}_{vvv} \mathbf{f}_{uuu}}{\left(1+\mathbf{f}_v^2+\mathbf{f}_u^2\right)^2} \\
                  \frac{-\mathbf{f}_{uvv} \mathbf{f}_{uuv}+\mathbf{f}_{vvv} \mathbf{f}_{uuu}}{\left(1+\mathbf{f}_v^2+\mathbf{f}_u^2\right)^2} & \frac{-2 \mathbf{f}_{uvv}^2+2 \mathbf{f}_{vvv} \mathbf{f}_{uuv}}{\left(1+\mathbf{f}_v^2+\mathbf{f}_u^2\right)^2}  \\
                \end{array}
              \right)$$

Because of \eqref{seconde}, we obtain that $\nabla \mathcal{K}(O)=0$.
In order to prove \eqref{h0}, for $w=r e^{it}$ we have  $$Q(t)=-\lim_{r\to 0}\frac{\mathcal{K}(\mathbf{w})}{|w|^2}=\left<\mathcal{\mathcal{H}}[\mathcal{K}](O)e^{it}, e^{it} \right>.$$
We need to show that $$R(t)=\frac{Q(t)}{1+\left<\nabla \mathbf{f}(0),e^{it} \right>^2}$$ does not depend on $t$.

Thus, by \eqref{seconde},
\begin{equation}\label{miami}Q(t)=\frac{\mathcal{A}+\mathcal{B}}{\left(1+\mathbf{f}_{v}^2+\mathbf{f}_{u}^2\right)^2}\end{equation}
where $$\mathcal{A}=-2 \sin^2 t \mathbf{f}_{uvv}^2+2 \sin t \left(\sin t \mathbf{f}_{vvv}-\cos t \mathbf{f}_{uvv}\right) \mathbf{f}_{vvu}$$ and
$$\mathcal{B}=-2 \cos^2 t \mathbf{f}_{vvu}^2+2 \cos t \left(\sin t \mathbf{f}_{vvv}+\cos t \mathbf{f}_{uvv}\right) \mathbf{f}_{uuu}.$$
By differentiating w.r.t. $u$ and $v$ the minimal surface equation we get
$$-2 \mathbf{f}_u \mathbf{f}_{uv}^2+\left(1+\mathbf{f}_u^2\right) \mathbf{f}_{uvv}+2 \mathbf{f}_{vv} \mathbf{f}_u \mathbf{f}_{uu}-2 \mathbf{f}_v \mathbf{f}_u \mathbf{f}_{uuv}+\left(1+\mathbf{f}_v^2\right) \mathbf{f}_{uuu}=0,$$
$$\mathbf{f}_{vvv} \left(1+\mathbf{f}_u^2\right)-2 \mathbf{f}_v \mathbf{f}_{uv}^2-2 \mathbf{f}_v \mathbf{f}_u \mathbf{f}_{uvv}+2 \mathbf{f}_v \mathbf{f}_{vv} \mathbf{f}_{uu}+\left(1+\mathbf{f}_v^2\right) \mathbf{f}_{uuv}=0.$$
Therefore
\begin{equation}\label{eq:uuu}\mathbf{f}_{uvv}=\frac{2 \mathbf{f}_v \mathbf{f}_{vvv} \left(\mathbf{f}_u+\mathbf{f}_u^3\right)+\left(1+\mathbf{f}_v^2\right)^2 \mathbf{f}_{uuu}}{-1-\mathbf{f}_u^2+\mathbf{f}_v^2 \left(-1+3 \mathbf{f}_u^2\right)}\end{equation}

and \begin{equation}\label{eq:vvv}\mathbf{f}_{uuv}=\frac{\mathbf{f}_{vvv} \left(1+\mathbf{f}_u^2\right)^2+2 \left(\mathbf{f}_v+\mathbf{f}_v^3\right) \mathbf{f}_u \mathbf{f}_{uuu}}{-1-\mathbf{f}_u^2+\mathbf{f}_v^2 \left(-1+3 \mathbf{f}_u^2\right)}.\end{equation}
Then inserting to \eqref{miami} we get
$$Q(t)=\frac{Y(t) \left(\mathbf{f}_{vvv}^2 \left(1+\mathbf{f}_{u}^2\right)^3+2 \mathbf{f}_{v} \mathbf{f}_{vvv} \mathbf{f}_{u}
\left(\mathbf{f}_{v}^2 \left(3-\mathbf{f}_{u}^2\right)+3 \left(1+\mathbf{f}_{u}^2\right)\right) \mathbf{f}_{uuu}+\left(1+\mathbf{f}_{v}^2\right)^3 \mathbf{f}_{uuu}^2\right)}{\left(\left(1+\mathbf{f}_{v}^2+\mathbf{f}_{u}^2\right)^2 \left(1+\mathbf{f}_{u}^2+\mathbf{f}_{v}^2 \left(1-3 \mathbf{f}_{u}^2\right)\right)^2\right)},$$
where
$$Y(t)=-2\left(1+\sin^2 t \mathbf{f}_{v}^2+ \sin (2t)  \mathbf{f}_{u}  \mathbf{f}_{v}+\cos^2 t \mathbf{f}_{u}^2\right)=-2(1+\left<\nabla\mathbf{f}(0),e^{it}\right>^2).$$
Since $$Y'(t)=-4 \cos(2t)\mathbf{f}_{v} \mathbf{f}_{u}+2 \sin (2t) \left(-\mathbf{f}_{v}^2+\mathbf{f}_{u}^2\right)$$ we get that $Y'(t)=0$ if and only if $\mathbf{f}_{v}=\mathbf{f}_{u}=0$.
In this particular case we get
$$-Q(0)=2(\mathbf{f}^2_{uuu}+\mathbf{f}^2_{vvv}).$$
\end{proof}

\subsection{Proof of Theorem~\ref{th:teo}}
The proof is based on Lemma~\ref{mainlema} and R. Hall sharp result best lower bound of the gradient of a harmonic diffeomorphism of the unit disk onto itself \cite{Hall1982}.
\begin{proof}[Proof of Theorem~\ref{th:teo}]

Assume that $w=f(z)=u(z)+\imath v(z)$ is the harmonic diffeomorphisms that produces the Enneper-Weisstrass parameterization so that $f(0)=0$.

The unit normal at $\mathbf{w}=(u,v,t)\in S$, in view of \cite[p.~169]{Duren2004}  is given by $$\mathbf{n}_{\mathbf{w}}=-\frac{1}{1+|q(z)|^2}(2\Im q(z), 2\Re q(z), -1+|q(z)|^2).$$ It is also given by the formula
$$\mathbf{n}_{\mathbf{w}}=\frac{1}{\sqrt{1+\mathbf{f}_u^2+\mathbf{f}_v^2}}\left(-\mathbf{f}_u,-\mathbf{f}_v,1\right).$$  Then we have the relations
 \begin{equation}\label{firsti}\mathbf{f}_v (u(x,y),v(x,y))=\frac{2 a(x,y)}{-1+a(x,y)^2+b(x,y)^2}
  \end{equation}
  \begin{equation}\label{secondi}\mathbf{f}_u(u(x,y),v(x,y))=\frac{2 b(x,y)}{-1+a(x,y)^2+b(x,y)^2},\end{equation} where $a=\Re q$, $b=\Im q$.
 So   \begin{equation}\label{terci}\nabla\mathbf{f}(u(x,y),v(x,y))=\frac{-2\imath \overline{q(z)}}{1-|q(z)|^2}\end{equation}
and
$$\mathcal{K}''(O)=\lim_{r\to 0}\frac{\mathcal{K}(rh)}{-2r^2(1+\left<\nabla \mathbf{f}(rk),k\right>^2)}.$$
Recall that in the Enneper-Weiersstrass parameters the curvature can be expressed as $$-\mathcal{K}(\mathbf{w}(z))=\frac{4|q'|^2}{|p|^2(1 + |q|^2)^4},$$ where $\mathbf{w}(z)=(f(z),t(z))\in S$.
 So
 \[
 \begin{split}
 \mathcal{K}''(O)&=\lim_{r\to 0}\frac{-\mathcal{K}(rk)}{|rk|^2(1+\left<\nabla \mathbf{f}(rk),k\right>^2)}
 \\&=\lim_{r\to 0}\frac{-\mathcal{K}(rk)}{|f^{-1}(rk)|^2}\frac{{|f^{-1}(rk)|^2}}{|rk|(1+\left<\nabla \mathbf{f}(rk),k\right>^2)}
 \\&=\frac{4|q''(0)|^2}{|p(0)|^2(1+|q(0)|^2)^4}\lim_{r\to 0}\frac{M^2(hr)}{(1+\left<\nabla \mathbf{f}(rk),k\right>^2)},
 \end{split}
 \]
   where
 $$M(kr)=\frac{1}{J(f,z)}|\overline{g'}k-\overline{h'}\overline{k}|=\frac{1}{|g'(z)|^2(1-|q(z)|^4)}|g'(z)|\cdot |1-q(z)^2 k^2|.$$ Here $z=f^{-1}(rk)$ and $J(f,z)$ is the Jacobian. Thus

 $$\lim_{r\to 0}M(kr)^2=\frac{|1-q(0)^2 k^2|^2}{|p(0)|^2(1-|q(0)|^4)^2}.$$
 Further
 \[\begin{split}\frac{|1-q(0)^2 k^2|^2}{(1+\left<\nabla \mathbf{f}(0),k\right>^2)}
&=\frac{|1-q(0)^2 k^2|^2}{1+\left<\frac{-2\imath \overline{q(0)}}{1-|q(0)|^2},k\right>^2} = (1-|q(0)|^2)^2.\end{split}\]
Thus \begin{equation}\label{eq:secderiv}\mathcal{K}''(O)=\frac{4|q''(0)|^2(1-|q(0)|^2)^2}{|p(0)|^4(1+|q(0)|^2)^4(1-|q(0)|^4)^2} =\frac{4|q''(0)|^2}{|p(0)|^4(1+|q(0)|^2)^6}.\end{equation}
Moreover, since $|q(z)|\le 1$, and $q'(0)=0$, we get $$\left|\frac{q(z)-q(0)}{1-q(z)\overline{q(0)}}\right|\le |z|^2$$ which implies that $|q''(0)|\le 2(1-|q(0)|^2)$.
Observe that $p(0)=f_z(0)$, and so by Hall's result (\cite{Hall1982}) $$|p(0)|^2\ge \frac{27}{4\pi^2}\frac{1}{1+|q(0)|^4}.$$
Therefore $$\mathcal{K}''(O)\le \frac{16(1-|q(0)|^2)^2}{|p(0)|^4(1+|q(0)|^2)^6}\le \frac{16\cdot (4\pi^2)^2}{27^2}.$$
Thus \begin{equation}\label{baca}\mathcal{K}''(O)\le \frac{256 \pi ^4}{729},\end{equation} what we wanted to prove.

\end{proof}

The following example produces the generalized  Scherk type surface for the second derivative of the Gaussian curvature above the center of the unit disk. We will show that it is extremal for the second derivative of Gaussian curvature (see the proof of Theorem~\ref{mainresult}).

\begin{example}\label{ex:hexagon}
Let $$g(z) = \frac{3 z F\left[\frac{1}{6},1,\frac{7}{6},-z^6\right]}{\pi }$$ and $$h(z) = \frac{3 z^5 F\left[\frac{5}{6},1,\frac{11}{6},-z^6\right]}{5 \pi }$$ where $F$ is the Gauss Hypergeometric function. Then
\begin{equation}\label{hexa}
f(z) = g(z) + \overline{h(z)}
\end{equation}
is a harmonic diffeomorphism of the unit disk onto the regular hexagon $\mathcal{G}$ inscribed in the unit disk.
Then $$p(z) =g'(z) = \frac{3}{\pi +\pi  z^6}$$ and $$q(z) = \sqrt{\frac{h'(z)}{g'(z)}}=z^2.$$

Moreover $f(z) = g(z) +\overline{h(z)}$ maps the unit disk onto the regular hexagon that defines a $6-$sides Scherk surface. See Figure~2.1.

The third coordinate of conformal harmonic parametrization is given by $$t(z) = \frac{\log\left[\frac{1+r^6-2 r^3 \sin(3s)}{1+r^6+2 r^3 \sin(3s)}\right]}{2 \pi }, z=r e^{is}.$$

\begin{figure}[htp]\label{f1}
\centering
\includegraphics{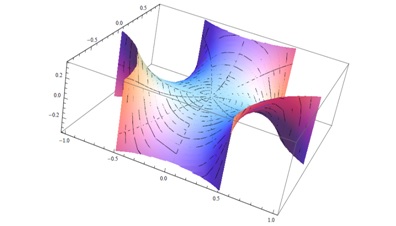}
\caption{A generalized Scherk surface over the hexagon.}
\end{figure}

See also the paper of Duren and Thygerson \cite{DurenThy} for additional details for those generalized Scherk surfaces.

Further \begin{equation}\label{eq:ine}|\mathcal{K}''(O)|=\frac{|q''(0)|^2}{|p(0)|^4}=\frac{16 \pi ^4}{81}.\end{equation} Observe that $\frac{16 \pi ^4}{81}<\frac{256 \pi ^4}{729}$ (see  \eqref{baca}).

\subsection{The proof of Theorem~\ref{mainresult}}
We need the following lemma.
\end{example}
\begin{lemma}\label{lema2}\cite[Lemma~1]{FinnOsserman1964}
 Let $\phi_1$, $\phi_2$ be any two distinct solutions of the minimal
surface equation such that $\phi_1$ and $\phi_2$, together with all their derivatives of
order up to and including $n$, coincide at some point $a$. Then a neighborhood
$D=D_a$ of that point $a$ may be mapped homeomorphically onto a neighborhood of $0$ in the
complex $\zeta-$ plane in such a way that the function $\phi=\phi_1-\phi_2$,  will be of the form
$$\phi(z)= \Re\left[\zeta(z)^N\right]$$ for some $N \ge  n + 1$, $z\in D_a$
\end{lemma}

\begin{proof}[Proof of Theorem~\ref{mainresult}]
Assume that $S:\zeta=\mathbf{f}(w)$ is a minimal surface above the unit disk with a zero-Gaussian curvature and horizontal tangent plane at the point above the center  and let $S^\diamond: \zeta=\mathbf{f}^\diamond(u,v)$ be the Scherk type minimal surface above the regular hexagon constructed in Example~\ref{ex:hexagon}. We assert that $|\mathcal{K}''(O)|<|\mathcal{K}_\diamond''(O)|$.
 Assume the converse $|\mathcal{K} ''(O)|\ge |\mathcal{K}_\diamond''(O)|$ and argue by a contradiction.  Then as in \cite{FinnOsserman1964}, by using the dilatation $L(\zeta) = \lambda \zeta$ for some $\lambda\ge 1$ we get the surface
$$S_\lambda=L(S)=\{(u,v,\lambda \mathbf{f}\left(\frac{u}{\lambda}, \frac{v}{\lambda}\right): |u+\imath v|<{\lambda}\},$$ whose Gaussian curvature
\begin{equation}\label{kk1}\mathcal{K}_\lambda(\mathbf{w})=\frac{1}{\lambda^2}\frac{ \left(\mathbf{f}_{uu}(u/\lambda,v/\lambda)\mathbf{f}_{vv}(u/\lambda,v/\lambda)-\mathbf{f}_{uv}(u/\lambda,v/\lambda)^2\right)}{(1+\mathbf{f}_u(u/\lambda,v/\lambda)^2
+\mathbf{f}_v(u/\lambda,v/\lambda)^2)^2}.\end{equation} Here and in the sequel $\mathbf{w}=(w,\mathbf{f}(w))$. Observe that such transformation does not change the gradient $\nabla \mathbf{f}(0)$. Moreover  $\mathcal{K}_\lambda(0)=0$. Since $$|{\mathcal{K}}''(O)|=\lim_{w\to 0}\frac{-\mathcal{K}(\mathbf{w})}{|w|^2}>0,$$ because of \eqref{kk1}, it exists $\lambda_\ast\ge 1$ so that \begin{equation}\label{kkad}\mathcal{K}_{\lambda_\ast}''(O)=\mathcal{K}_\diamond''(O).\end{equation}

 Let $$\mathbf{f}^\ast (u,v) =\lambda_\ast \mathbf{f}\left(\frac{u}{\lambda_\ast}, \frac{v}{\lambda_\ast}\right).$$

 Since $K_\ast(O)=\mathcal{K}_{\diamond}(O)=0$ and $\nabla \mathbf{f}^\ast(0) =0= \nabla \mathbf{f}^\diamond(0)$, it follows that all derivatives up to the order $2$ of $\mathbf{f}_\ast$ and $\mathbf{f}_\diamond$ vanishes at zero.

From \eqref{kkad} we obtain that \begin{equation}\label{combined}(\mathbf{f}^\ast_{uuu})^2+(\mathbf{f}^\ast_{vvv})^2=(\mathbf{f}^\diamond_{uuu})^2+(\mathbf{f}^\diamond_{vvv})^2.\end{equation}
We can also w.l.g. assume that $\mathbf{f}^\ast_{uuu} $ and $\mathbf{f}^\diamond_{uuu} $ as well as $\mathbf{f}^\ast_{vvv} $ and $\mathbf{f}^\diamond_{vvv} $ have the same sign. If not, then we choose $\lambda_\ast\le -1$ and repeat the previous procedure with $$S_1=L(S)=\{(u,v,\lambda \mathbf{f}\left(\frac{u}{\lambda}, \frac{v}{\lambda}\right): |u+\imath v|<{|\lambda|}\}.$$

From \eqref{h0}, \eqref{eq:uuu}, \eqref{eq:vvv} and \eqref{combined} we obtain that $\mathbf{f}^\ast_{uuu}=\mathbf{f}^\diamond_{uuu}$, $\mathbf{f}^\ast_{vvv}=\mathbf{f}^\diamond_{vvv}$, $\mathbf{f}^\ast_{uuv}=\mathbf{f}^\diamond_{uuv}$, $\mathbf{f}^\ast_{uvv}=\mathbf{f}^\diamond_{uvv}$.

Thus the function $F(u,v) = \mathbf{f}^\ast(u,v)- \mathbf{f}^\diamond(u,v)$ has all derivatives up to the order $3$ equal to zero in the point $w=0$.
%Then the modified surface with the angle $t$ satisfies

By Lemma~\ref{mainlema}, for $F(w) = \mathbf{f}^\ast(w)-\mathbf{f}^\diamond(w)$ we are in situation of
 Lemma~\ref{lema2}, with $n = 3$. We conclude that there exists a homeomorphism
of a neighborhood of the origin onto a neighborhood of the origin
in the  $\zeta$ plane such that $F(z) = \Re \{\zeta(z)^N\}$ for $N \ge  4$. In particular, the level locus
$F=0$ in this neighborhood consists of $N$ arcs intersecting only at the origin
which divide the neighborhood into $2N$ sectors in which $F$ is alternately positive
and negative.

Let us now examine any component $R$ of the open set of points in the hexagon  $\mathcal{G}$ at
which $F\neq 0$. At a boundary point of $R$ which is interior to $\mathcal{G}$ we must have
$F= 0$. Thus the boundary of $R$ cannot consist entirely of intericr points of $\mathcal{G}$,
since for the difference of two solutions the maximum principle holds \cite[p.~322]{CourantHilbert1962}
and hence we would have $F = 0$ in $R$. Thus the boundary of $R$ must contain
at least one point of the boundary of $\mathcal{G}$. If it contains an inner point of a side
of $\mathcal{G}$ then it must contain all points on that side, since on each side $\mathbf{f}^\diamond\to  \pm \infty$
but $\mathbf{f}^\ast$ is bounded. Thus the set $F=0$ has at most eight  components whose
boundaries contain inner points of a side of $\mathcal{G}$. On the other hand, the set
$F\neq 0$ has at least 6 components, since this is true at the origin and otherwise
we could find an arc lying in a single component $R$ which joined two different
sectors at the origin. Then one of the other sectors at the origin would lie in a
component whose boundary was entirely interior to $\mathcal{G}$ which, as we have seen,
is impossible. We therefore conclude that some component $R$ of $F\ne 0$ has as
boundary points only interior points of $\mathcal{G}$ and one or more of the vertices of $\mathcal{G}$.
However, the  Finn's maximum principle applies also in this case (see \cite{FinnR1963}), and we would again find $F=0$ in $R$. Thus the assumption that the surface
$S^{\lambda_\ast}$ has the same second derivative of the Gaussian curvature  at the origin as Scherk's generalized surface $S^\diamond$ and can be extended
across the sides of $\mathcal{G}$ leads to a contradiction, and in view of \eqref{eq:ine},  inequality \eqref{interesting} is proved.

Prove the sharpness  part. A similar statement for the extremal Gaussian curvature at the centre for the case that the tangent plane is horizontal has been proved in \cite[Proposition~3]{FinnOsserman1964}. However that proof does not work in this case.
Assume that  $\omega= z^4$. Then $f$ defined in \eqref{hexa} is a solution of  Beltrami equation $\overline{f}_z=\omega f_z$ satisfying the initial conditions $f_z(0)>0$ and $f(0)=q(0)=q'(0)=0$. Further $f$ maps the unit disk onto the regular hexagon $\mathcal{G}$.

Furthermore, for $0<k<1$ assume that  $\omega_k=k^2 \omega.$  Then solve the second Beltrami equation $\overline{f}_z=\omega_k f_z$ that map the unit disk $\D$ onto itself  satisfying the initial condition $f(0) = 0$ and $f_z(0)>0$ \cite{HengartnerSchober1986}. This mapping exists and is unique \cite[p.~134]{Duren2004}. Then this mapping produces a minimal surface $S_k$ over the unit disk. Moreover for $k=n/(n+1),$ the sequence $f_n$ converges (up to some subsequence) in compacts of the unit disk, to a mapping $f^\circ$ that maps the unit disk into the unit disk. By using again the uniqueness theorems  \cite[Theorem~B\&~Theorem~1]{zbMATH05159460}, because $f^\circ(0)= f(0)$ and $f^\circ_z(0)>0$, it follows that $f^\circ \equiv f$.

Let $\mathbf{w}_n$ be the point above $0$ of minimal surface $S^n=S_{n/(n+1)}$. Then $f_n(0)=0$. Moreover the second derivative of the Gaussian curvatures $\mathcal{K}_n(\mathbf{w}_n)$ of $S^n$, in view of the formula \eqref{eq:curvatureformula}, is equal to \begin{equation}\label{eq:secderiv}\mathcal{K}_n''(\mathbf{w}_n) =\frac{4|q_n''(0)|^2}{|p_n(0)|^4(1+|q_n(0)|^2)^6}\end{equation} and converges to the second derivative of the  Gaussian curvature $\mathcal{K}''(\mathbf{w}).$ To prove the last fact, observe that $f_n$ and also $f$ are quasiconformal in a disk around $0$ and the family is normal. This is why $q_n$ and $p_n$ and $q_n''$ converges in compacts to the corresponding $q$, $p$ and $q''$. This implies that \eqref{interesting} cannot be improved

\end{proof}

%%%%%%%%%%
%%%%%%%%%%
%%%%%%%%%%
%%%%%%%%%%   THE BIBLIOGRAPHY
%%%%%%%%%%
%%%%%%%%%%

{\bibliographystyle{abbrv} \bibliography{references}}

%%%%%%%%%%
%%%%%%%%%%
%%%%%%%%%%
%%%%%%%%%%   AFFILIATIONS
%%%%%%%%%%
%%%%%%%%%%

\vspace*{3mm}
\noindent David Kalaj

\noindent University of Montenegro, Faculty of Natural Sciences and Mathematics, 81000, Podgorica, Montenegro

\noindent e-mail: {\tt davidk@ucg.ac.me}

\end{document}